\documentclass[12pt, a4paper]{article}
\usepackage{geometry, microtype}
\usepackage{titlesec} 
\usepackage{amsfonts}
\usepackage{amsthm}
\usepackage{amssymb}
\usepackage{amsmath}
\usepackage{theoremref}
\usepackage{enumerate}
\usepackage{bbm}
\usepackage{bm}
\usepackage{mathtools}
\usepackage{xcolor}

\makeatletter
\newcommand{\proofpart}[2]{%
  \par
  \addvspace{\medskipamount}%
  \noindent\emph{Step #1: #2}\par\nobreak
  \addvspace{\smallskipamount}%
  \@afterheading
}
\makeatother

\DeclarePairedDelimiter\abs{\lvert}{\rvert}%
\DeclarePairedDelimiter\norm{\lVert}{\rVert}%

\makeatletter
\let\oldabs\abs
\def\abs{\@ifstar{\oldabs}{\oldabs*}}
\let\oldnorm\norm
\def\norm{\@ifstar{\oldnorm}{\oldnorm*}}
\makeatother

\makeatletter
\g@addto@macro\bfseries{\boldmath}
\makeatother

\newcommand{\T}{\mathbb{T}}

\newcommand{\Ka}{\mathcal{K}}
\newcommand{\conj}[1]{\overline{#1}}
\newcommand{\R}{\mathbb{R}}
\newcommand{\supp}[1]{\text{supp}({#1})}

\DeclareMathOperator{\Caplp}{Cap_{\ell^\mathit{p}}}

\newtheorem{thm}{Theorem}[section]
\newtheorem{lemma}[thm]{Lemma}
\newtheorem{cor}[thm]{Corollary}
\newtheorem{prop}[thm]{Proposition}
\theoremstyle{definition}

\newcommand{\Addresses}{{
    \bigskip
    \footnotesize
    \noindent
    Adem Limani, \\ \textsc{Lund University, Centre for Mathematical Sciences,  \\
      Box 118, 221 00 Lund, Sweden}\\
    \texttt{adem.limani@math.lu.se}
    
    \medskip
    \noindent
    Tomas Persson, \\ \textsc{Lund University, Centre for Mathematical Sciences, \\
      Box 118, 221 00 Lund, Sweden}\\
    \texttt{tomasp@gmx.com}		
  }}

\begin{document}
\title{\textbf{Asymmetric uniqueness sets in $\ell^q$}}

\author{Adem Limani \& Tomas Persson} 

\date{\today}


\maketitle

\begin{abstract}
  \noindent We exhibit an asymmetry phenomenon for uniqueness
  sets in $\ell^q$. Specifically, we construct sets that do not
  support measures with $\ell^q$-summable Fourier coefficients,
  yet simultaneously support measures whose positive frequencies decay faster than polynomials. In
  the language of Fourier uniqueness, this highlights a striking
  divergence between the unilateral and bilateral $\ell^q$
  uniqueness problems.
\end{abstract}

\section{Introduction}
\subsection{Unilateral versus bilateral Fourier decay}

Let $\T \cong [0,1)$ denote the unit circle and let $M(\T)$ be
the space of complex finite Borel measures on $\T$.  The purpose
of this paper is to demonstrate a pronounced asymmetry between
unilateral and bilateral Fourier decay at the level of supporting
sets.

For $1<q<2$, we construct compact sets $E \subset \T$ of Lebesgue
measure arbitrarily close to $1$ with the following property:

\begin{quote}
  $E$ supports a nontrivial measure $\mu \in M(\T)$ whose one-sided sequence of
  Fourier coefficients $\{\widehat{\mu}(n)\}_{n>0}$ belong to $\ell^q$, yet $E$
  supports no nontrivial measure $\mu \in M(\T)$ whose two-sided sequence of
  Fourier coefficients $\{\widehat{\mu}(n)\}_{n\in \mathbb{Z}}$ belongs to $\ell^p$ for any
  $0<p<2$.
\end{quote}

Recall that a compact set $E$ is called a \emph{uniqueness set
  for $\ell^q$} if it supports no nontrivial
$\mu \in M(\T)$ whose Fourier coefficients lie in
$\ell^q$:
\[
\sum_{n \in \mathbb{Z}} \abs{\widehat \mu (n)}^q < \infty.
\]
We denote by $M(E)$ the subset of $M(\T)$ consisting of measures
supported in a compact set $E\subset\T$. In the regime
$0<q\le 2$, any measure with $\ell^q$-summable Fourier
coefficients is absolutely continuous with respect to the
unit-normalized Lebesgue measure $dm$ on $\T$, with density in
$L^2(\T,dm)$.

Our main theorem exhibits a substantially stronger asymmetry
phenomenon than the aforementioned result.

\begin{thm}\thlabel{THM:MAIN}
There exists a compact set $E\subset \T$ of Lebesgue measure arbitrarily close to $1$ such that:
\begin{enumerate}
\item[(i)] Every nontrivial complex Borel measure $\mu$ with
  support in $E$ satisfies
  \[
    \sum_{n\in\mathbb Z} |\widehat\mu(n)|^q = \infty,
    \qquad 0<q<2.
  \]
\item[(ii)] There exists $f\in L^\infty(\T,dm)$ supported in $E$
  such that for every $M>0$ there exists $C(M)>0$ with
  \[
    |\widehat f(n)| \le C(M) n^{-M},
    \qquad n=1,2,3,\ldots .
  \]
\end{enumerate}
\end{thm}

The exponent $q=2$ is the critical threshold, since Parseval's
identity ensures that the indicator function of every set of
positive Lebesgue measure has Fourier coefficients in
$\ell^2$.  The theorem therefore shows that one may
approach the $\ell^2$ threshold arbitrarily closely
from below while retaining a substantial unilateral–bilateral
asymmetry on the same supporting set. A non-periodic analogue in
the setting of $L^q(\R)$ is obtained in Section~5.

Uniqueness sets for $\ell^q$ were constructed in the
classical works of D.~Newman \cite{newman1964closure},
Y.~Katznelson \cite{katznelson1964sets}, and
I.~Hirschman--Y.~Katznelson \cite{hirschman1965sets}. See also
J.~Rosenblatt and K.~Shuman in \cite{rosenblatt2003cyclic}. In
those constructions, however, no quantitative control on the
Beurling--Carleson entropy (see Theorem~\ref{THM:KHRUSHMAIN} for
definition) is available.

A central novelty of the present work is to device a new
construction which makes simultaneous and explicit control of
Fourier concentration and Beurling--Carleson entropy.

In \thref{THM:MAIN} we constructed supporting sets which allow
for exceptional \emph{unilateral} Fourier decay, while forcing
bilateral decay to remain at the $\ell^2$-threshold. Our next result
exhibits the complementary phenomenon, allowing us to construct
sets that support bilateral decay arbitrarily close to
$\ell^1$, yet prohibit \emph{any} prescribed uniform
unilateral Fourier decay.

\begin{thm}\thlabel{THM:MAIN2}
Let $\{\Omega(n)\}_n$ be positive 
real numbers with $\Omega(n) \downarrow 0$.  Then there exists a compact sets $E\subset \T$ of Lebesgue measure arbitrarily close to $1$ such that:

\begin{enumerate}
\item[(i)] 
There exists a non-negative function $f\in L^2(\T)$ supported on $E$
whose Fourier coefficients satisfy
\[
  \{\widehat f(n)\}_{n\in \mathbb{Z}} \in \bigcap_{r>1} \ell^r .
\]

\item[(ii)]
If $\mu\in M(E)$ satisfies the uniform unilateral estimate
\[
|\widehat\mu(n)|\le \Omega(n),
\qquad n=1,2,3,\ldots,
\]
then $\mu\equiv 0$.
\end{enumerate}
\end{thm}

\subsection{Background}

The study of unilateral versus bilateral spectral behaviour has
long historic roots in harmonic analysis. A classical theorem of
A.~Rajchman \cite{rajchman1929classe} asserts that for any $\mu \in M(\T)$,
\[
\widehat\mu(n)\to0 \ \text{as } n\to+\infty
\quad\Longrightarrow\quad
\widehat\mu(n)\to0 \ \text{as } |n|\to\infty.
\]
Quantitative refinements and related symmetry principles were
obtained by K.~de~Leeuw and Y.~Katznelson \cite{deleeuw1970two}
and by J.-P.~Kahane and Y.~Katznelson \cite{kahane1970algebres}.
For broader perspective on so-called Rajchman measures and
uniqueness phenomena, see the surveys of T.~W.~K\"orner
\cite{korner1992sets} and R.~Lyons \cite{lyons2020seventy}.

More delicate weighted symmetry phenomena were discovered by
S.~Khrushchev and W.~Peller \cite{peller1982hankel}, who proved
that
\[
\sum_{n>0} \frac{|\widehat\mu(n)|^2}{n+1} < \infty
\quad\Longrightarrow\quad
\sum_{n\in\mathbb Z} \frac{|\widehat\mu(n)|^2}{|n|+1} < \infty.
\]
This result relies on deep best approximation techniques, that
allow one to reconstruct measures from their Cauchy transforms,
see \cite[Theorem~3.16]{peller1982hankel}. However, as was noted
by G.~Kozma and A.~Olevskii \cite{kozma2013singular}, this
results fails for fractional weights $(n+1)^{-s}$ with $0<s<1$.

As the aforementioned symmetry-type results on unilateral to
bilateral spectral decay have only be prevalent in few special
cases, it is from the perspective of sets of uniqueness in
Fourier analysis, more convenient to study the behaviours of
supporting sets. More specifically, we may ask whether a compact
set $E \subset \T$ supports a measure (or a distribution) with
certain unilateral Fourier decay, also supports a measure
(distribution) with the same, but now bilateral Fourier decay?
As we saw from the work G.~Kozma and A.~Olevskii, even though the
corresponding symmetry problem breaks down in the framework of
fractional Dirichlet--Sobolev spaces $\ell^{2,-\alpha}$:
\[
  \sum_{n \in \mathbb{Z}} (1+|n|)^{-\alpha} \abs{\widehat{\mu}(n)}^2 < \infty
\]
for $\alpha\in (0,1)$, one can still show that if $E$ supports a
measure with unilateral positive Fourier coefficients in
$\ell^{2,-\alpha}$, then it also supports a
probability measures with the same property, and whose Fourier
coefficients (both unilateral and bilateral) are in
$\ell^{2,-\alpha}$. On the other hand, S.~Khrushchev showed that
there exists a compact set $E \subset \T$ which supports a
measure $\mu$ with unilateral Fourier coefficients in
$\ell^{2,1}$, but does not support a non-trivial
measure with bilateral Fourier coefficients in
$\ell^{2,1}$. These results depend on the specific
potential theoretical framework on $\ell^{2,1}$, and using
further refinements of these techniques, N.~G.~Makarov
\cite{makarov1991class} proved similar results for
$\ell^{2,\alpha}$ with $0<\alpha<1$. In similar fashion, it was
recently proved by Kozma and Olevskiǐ \cite{kozma2003null} that
there exists a non-trivial distribution $S$ supported on a set of
Lebesgue measure zero, such that
$\sum_{n>0} |\widehat{S}(n)|^2 < \infty$, but clearly
$\{\widehat{S}(n)\}_n \notin \ell^2$ in view of
Parseval's Theorem. In the setting of $\ell^q$, for $q>2$, N.~Lev
and A.~Olevskii \cite{lev2011wiener} constructed compact sets
supporting distributions with $\ell^p$-summable Fourier
coefficients but no such measures, a development that ultimately
led to them disproving Wiener's conjecture on describing cyclic
elements only by zero sets.

Our principal contribution is to study the corresponding problem
in the setting of $\ell^{q}$ for $q<2$.

\subsection{Uniqueness sets and Fourier capacity}

In addition to the lack symmetry of supporting sets for
unilateral and bilateral Fourier decay, we develop a structural
characterization of $\ell^q$-uniqueness sets via a
natural Fourier capacity. For a set $E \subset \mathbb{T}$, we
let $M(E)$ denote the set of complex Borel measures of finite
total variation and support in $E$.  Let $A_p(\T)$ denote the
space of distributions whose Fourier coefficients lie in
$\ell^p $, equipped with the norm
\[
\lVert f \rVert_{A_p (\mathbb{T})} = \lVert \{\widehat{f}(n)\}_{n\in \mathbb{Z}}
\rVert_{\ell^p}.
\]
For $1<q<2$ and $p=q/(q-1)$,
define the notion of $\ell^p$-capacity
\[
\Caplp (E)
=
\inf_{\substack{\phi=1 \text{ on } E\\ \phi\in C(\T)}}
\|\phi\|_{A_p (\mathbb{T})}.
\]

\begin{thm}\thlabel{THM:CHARlq}
  Let $1<q<2$ and $p=q/(q-1)$.  A compact set $E\subset\T$
  supports no nontrivial measure $\mu$ with
  $\{\widehat\mu(n)\}_{n\in \mathbb{Z}} \in \ell^q $ if and only if\/
  $\Caplp (E)=0$.  Moreover,
  \begin{equation}
    \label{eq:capacityformula}
    \Caplp (E)
    =
    \sup_{\substack{\nu\in M(E)\\ \|\nu\|_{A_q}\le1}}
    |\nu(E)|.
  \end{equation}
\end{thm}

While related mechanisms appear implicitly in earlier
constructions, the capacity formulation clarifies the metric
structure of uniqueness sets in $\ell^q $. Combining
\thref{THM:MAIN} and \thref{THM:CHARlq}, we obtain compact sets
of arbitrarily large measure that have finite Beurling--Carleson
entropy and yet zero $\ell^p$-capacity for all $p>2$.
From the works of Katznelson, Newman, et al., it is possible to
get the impression that uniqueness sets for $\ell^q$
must necessarily have infinity Beurling--Carleson entropy.  We
show that this is not the case.

A further consequence is a discrepancy between simultaneous
approximation by trigonometric and analytic polynomials.

\begin{cor}
  \thlabel{COR:SAlq}
  For any $p>2$, there exists compact sets $E \subset \T$ of
  Lebesgue measure arbitrary close to full, which satisfy the
  following properties:
  \begin{enumerate}
  \item[(a)] For any $f\in A_p(\T)$ and any continuous function
    $g$ on $\T$, there exists trigonometric polynomials
    $\{T_N\}_N$ with the properties that
    \[
      \sum_{n \in \mathbb{Z}} |\widehat{T}_N(n) -\widehat{f}(n)|^p \to 0, \qquad
      \sup_E \abs{T_N -g} \to 0.
    \]
  \item[(b)] Whenever $\{Q_N\}_N$ are analytic polynomials which
    satisfy the properties:
    \[
      \sum_{n \in \mathbb{Z}} \, |\widehat{Q}_N(n)-\widehat{f}(n)|^p \to 0, \qquad
      \sup_{E} \abs{Q_N} \to 0,
    \]
    for some $f\in H^2$, then $f$ is identically zero.
\end{enumerate}
\end{cor}

Here, $H^p$ denotes the classical Hardy spaces, for $p\geq 1$. Each Hardy space $H^p$ may be isometrically and isomorphically identified with the closed subspace of $L^p (\T, dm)$ consisting of elements with $\widehat{f} (n) = 0$ for $n < 0$.

Here (a) is a simultaneous approximation phenomenon for
trigonometric polynomials, whereas (b) is a Khinchin--Ostrowski
type rigidity statement on analytic functions, see Havin and
J\"{o}ricke \cite[Chapter 3, §2]{havinbook} and the further work of Khrushchev surveyed
therein. Related simultaneous approximation phenomenons for
$\ell^q$ in the range $1<q<2$ were considered already
by Kahane--Katznelson \cite{kahane1971comportement}, involving
intricate probabilistic constructions. Based on those techniques,
analytic variants appeared in Kahane--Nestoridis
\cite{kahane2000series}. A different construction involving inner
functions was recently announced by A. Limani
\cite{limani2025fourier}.

\subsection{Method and organization}

The construction underlying \thref{THM:MAIN} combines three
principal ingredients.  First, we introduce explicit frequency
blocks whose arithmetic separation forces any supported measure
to accumulate $\ell^q $-mass.  Second, we maintain
quantitative control of the complementary arcs, thereby
preserving finite Beurling--Carleson entropy and enabling the use
of unilateral decay results of Khrushchev type.  Third, a duality
argument identifies $\ell^q$-uniqueness with
vanishing Fourier capacity, clarifying the metric structure
underlying the phenomenon.

Section~2 develops the Beurling--Carleson entropy framework and recalls the necessary unilateral decay results. 
Section~3 begins with the Fourier capacity characterization and its connection to simultaneous approximation; the remainder of the section is devoted to the construction of the exceptional sets and the proof of \thref{THM:MAIN}.

Section~4 contains our constructions of non-uniqueness sets, and uniqueness sets of uniform type, which combined prove \thref{THM:MAIN2}.

In Section~5 we exhibit non-periodic $L^q(\R)$-analogues of our main results.


\section{Entropy and unilateral Fourier decay}

\subsection{Unilateral polynomial decay and Beurling--Carleson entropy}

We recall the entropy condition governing the existence of measures with smooth Cauchy transforms. 
This goes back to the pioneering work of S.~Khrushchev, who established the following fundamental characterization.

\begin{thm}[Khrushchev, 1973 {\cite{khrushchev1978problem}}] \thlabel{THM:KHRUSHMAIN}
Let $E\subsetneq \T$ be compact. The following are equivalent:
\begin{enumerate}
\item[(a)] There exists a nontrivial finite complex Borel measure $\mu$ supported on $E$ whose Cauchy transform
\[
\mathcal K(\mu)(z)
=
\int_E \frac{d\mu(\zeta)}{1-\overline{\zeta}z},
\qquad |z|<1,
\]
extends to a $C^\infty$-function on $\T$.
\item[(b)] $E$ contains a compact subset $E_0$ of positive Lebesgue measure, which has finite Beurling--Carleson entropy:
\[
\mathcal E(E_0)
:=
\sum_j |I_j| \log\frac{1}{|I_j|}
<\infty,
\]
where $\{I_j\}$ are the connected components of\/ $\T\setminus E_0$.
\end{enumerate}
\end{thm}

Here $|I|$ denotes the length of an arc $I\subset\T$. A crucial
but simple remark is that $C^\infty$-regularity of the Cauchy
transform in (a) is equivalent to super-polynomial decay
of the positive Fourier coefficients of $\mu$, via the formula
\[
\mathcal \Ka(\mu)(z)=\sum_{n\ge0}\widehat\mu(n)z^n,
\qquad |z|<1.
\]
This formula shows that $\Ka(\mu)$ extends to a
$C^\infty$-function on $\T$, if and only if $\widehat{\mu} (n)$
decays faster than any polynomial as $n \to \infty$. In other
words, the Beurling--Carleson entropy of a subset determines the
capacity for the Cauchy integral to have smooth extensions to
$\T$.

In recent work by A. Limani and B. Malman
\cite{limani2021constructions}, it was proved that if $E$ has
finite Beurling--Carleson entropy, such a measure $\mu \in M(E)$
may in fact be constructed explicitly, and one may additionally
take $d\mu=f\,dm$ with $f\in L^\infty(E)$, the set of essentially
bounded functions supported on $E$.

We close this section with a simple sufficient condition for a
set of have finite Beurling--Carleson entropy, useful for our
purposes.

\begin{lemma}\thlabel{LEM:BC}
  Let
  \[
    E = \bigcap_j (\T\setminus U_j),
  \]
  be a compact subset of\/ $\T$, where each $U_j=\bigcup_k I_k(j)$
  is a finite union of open arcs. If
  \[
    \sum_{j,k} |I_k(j)| \log \frac{1}{|I_k(j)|}<\infty,
  \]
  then $E$ has finite Beurling--Carleson entropy.
\end{lemma}

\begin{proof}
  Let $\{J_\ell\}$ denote the connected components of
  $\T\setminus E$.  Since $\T\setminus E=\bigcup_{j,k} I_k(j)$,
  each $J_\ell$ is a disjoint union of arcs $I_k(j)$.  Using that
  $\log$ is increasing, we get
  \[
    \sum_\ell |J_\ell| \log \frac{1}{|J_\ell|} = \sum_\ell
    \sum_{I_k(j)\subset J_\ell} |I_k(j)| \log \frac{1}{|J_\ell|}
    \le \sum_{j,k} |I_k(j)| \log \frac{1}{|I_k(j)|}
  \]
  which proves the claim.
\end{proof}

\section{Sets of uniqueness in $\ell^q $}

\subsection{Characterizing uniqueness sets in $\ell^q$}

We begin with the proof of \thref{THM:CHARlq}, which provides a
structural characterization of uniqueness set for
$\ell^q $. The argument is self-contained and relies
only on duality and standard approximation arguments.

For $1\le q<\infty$, let $A_q(\T)$ denote the Banach space of
distributions $S$ on $\T$ whose Fourier coefficients satisfy
$\{\widehat S(n)\}_{n\in \mathbb{Z}} \in \ell^q $, equipped with the
norm
\[
  \|S\|_{A_q}:=\|\{\widehat S(n)\}_{n \in \mathbb{Z}}\|_{\ell^q
    }.
\]
When $1\le q\le2$, the space $A_q(\T)$ embeds continuously into
$L^2(\T,dm)$.

For $1<q<\infty$, the duality between $A_q(\T)$ and $A_p(\T)$, where $p=q/(q-1)$, is considered in the standard Fourier pairing
\[
\langle f,g\rangle
:=
\sum_{n \in \mathbb{Z}} \widehat f(n)\overline{\widehat g(n)}.
\]
Whenever, in addition, $fg\in L^1(\T)$, this pairing also admits the convenient integral representation
\[
\langle f,g\rangle
=
\int_{\T} f\overline g\,dm.
\]

The following proposition will serve as our foundation. 

\begin{prop}\thlabel{PROP:CHARlq}
  Let $1<q\le2$ and write $p=q/(q-1)$.  For a compact set
  $E\subset\T$ of positive Lebesgue measure, the following are
  equivalent:
  \begin{enumerate}
  \item[(a)] $E$ supports no nontrivial measure $\mu$ with
    $\{\widehat\mu(n)\}_{n \in \mathbb{Z}} \in \ell^q$.
  \item[(b)] $C(\T\setminus E)$, the set of continuous functions
    on $\T$ which vanish on $E$, is dense in
    $A_p(\T)$.
  \item[(c)] For every $f\in A_p(\T)$ and every $g\in C(E)$,
    there exist trigonometric polynomials $T_n$ such that
    \[
      \|T_n-f\|_{A_p}\to0,
      \qquad
      \sup_E |T_n-g|\to0.
    \]
  \item[(d)] $E$ has $\ell^p$-Fourier capacity zero:
    \[
      \Caplp (E)
      :=
      \inf_{\substack{\phi=1\text{ on }E \\ \phi\in C(\T) }}
      \|\phi\|_{A_p}
      =0.
    \]
  \end{enumerate}
\end{prop}

\begin{proof}
  (a)$\Rightarrow$(b).  If $C(\T\setminus E)$ were not dense in
  $A_p(\T)$, Hahn--Banach would yield a nonzero $g\in A_q(\T)$
  annihilating $C(\T\setminus E)$.  Thus $g$ vanishes off $E$,
  and $d\mu=g\,dm$ defines a nontrivial measure supported on $E$
  with $\widehat \mu \in \ell^q$, violating the statement in (a).

\medskip
\noindent
(b)$\Rightarrow$(c).  Consider the Banach space
$A_p(\T)\oplus C(E)$ with norm
$\|(f,g)\|=\|f\|_{A_p}+\|g\|_\infty$. As the trigonometric
polynomials $T$ are dense in each component separately, it
suffices to approximate $(1,0)$ by diagonal pairs $(T,T)$. Using
the assumption that $C(\T\setminus E)$ is dense in $A_p(\T)$,
we can for every $\varepsilon>0$ find
$\phi_\varepsilon\in C(\T\setminus E)$ such that
\[
\|\phi_\varepsilon-1\|_{A_p}\le\varepsilon.
\]
Passing to appropriate Fej\'er means yields a trigonometric polynomial $T_\varepsilon$ with
\[
\|T_\varepsilon-1\|_{A_p}\le2\varepsilon,
\qquad
\sup_E |T_\varepsilon|\le\varepsilon.
\]
This proves (c).

\medskip
\noindent
(c)$\Rightarrow$(d).
Observe that
\[
\Caplp (E)
=
\inf_{\phi\in C(\T\setminus E)}
\|1-\phi\|_{A_p},
\]
so the capacity equals the distance from $1$ to $C(\T\setminus E)$ in $A_p(\T)$.

If $\Caplp (E)>0$, then $C(\T\setminus E)$ is not
dense in $A_p(\T)$.  By the proof of (a)$\Rightarrow$(b), there
exists a nontrivial $\mu\in A_q(\T)$ supported on $E$. Now the
tuple $(\mu,-\mu) \in A_q(\T) \oplus M(E)$ regarded as a
functional on $A_q(\T) \oplus C(E)$ has the property that it
annihilates the diagonal set
\[
\{(\zeta^n, \zeta^n): n=0,\pm 1, \pm 2, \ldots \}
\]
in the customary dual-pairing $\left(A_p(\T) \oplus C(E) \right)' \cong A_q(\T) \oplus M(E)$. This shows that diagonal tuples of trigonometric polynomials $(T,T)$ cannot be dense in $A_p(\T) \oplus C(E)$.

\medskip
\noindent
(d)$\Rightarrow$(a).
If $\Caplp (E)=0$, then $C(\T\setminus E)$ is dense in $A_p(\T)$. 
Let $\mu\in M(E)\cap A_q(\T)$ and let $\phi_j\in C(\T\setminus E)$ with $\phi_j\to1$ in $A_p(\T)$. 
Then for every integer $n$,
\[
 \widehat{\mu}(n) = \lim_j \sum_k \widehat{\mu}(n-k) \conj{\widehat{\phi_j}(k)} = \lim_j \int_{\T} \conj{\phi_j(\zeta)} \zeta^{-n} d\mu(\zeta) =0
\]
hence $\mu\equiv0$.
\end{proof}

We remark that one can substitute the statement in $(b)$ by the
statement
\[
(b') \quad L^2(\T \setminus E) \text{ is dense in } A_p(\T),
\]
where $L^2(\T \setminus E)$ denotes the subset of $L^2(\T,dm)$
which vanish $dm$-a.e on $E$. An analogous characterization as in
\thref{PROP:CHARlq} also holds for $q>2$, but with appropriate
modifications taking into account that $A_q(\T)$ consists of
distributions in this range.

\begin{proof}[Proof of \thref{THM:CHARlq}]
  The first part follows by Proposition~\ref{PROP:CHARlq}.
  
  It remains to establish the distance formula
  \eqref{eq:capacityformula}. Let $X$ denote the closure of
  $C(\T\setminus E)$ in $A_p(\T)$.  Then
  \[
    \Caplp (E) = \inf_{\phi \in C(\T \setminus E)}
    \norm{1-\phi}_{A_p} =: \|1\|_{A_p/X}.
  \]
  where $A_p / X$ is the quotient Banach space, equipped with the
  natural norm being the distance to $X$ in the $A_p(\T)$. By
  duality, we have
  \[
    (A_p/X)'
    \cong
    X^\perp
    =
    \{g\in A_q(\T): \supp g\subseteq E\},
  \]
  interpreted in the sense of isomorphisms of Banach spaces. This
  yields the identity
  \[
    \Caplp (E) = \sup_{\substack{\|g\|_{A_q}\le1\\
        \supp g\subseteq E}} \left|\int_{\T} g\,dm\right| =
    \sup_{\substack{\|\nu_g\|_{A_q}\le1\\ \supp{\nu_g}\subseteq
        E}} \left|\nu_g(E)\right|,
  \]
  where $d\nu_g = g dm$. Since all $\nu \in A_q \cap M(E)$ are of
  this form, the proof is complete.
\end{proof}

\medskip

We conclude with the corollary distinguishing simultaneous approximation by trigonometric and analytic polynomials.

\begin{proof}[Proof of \thref{COR:SAlq}]
  We argue using \thref{THM:MAIN}, whose proof is independent of
  the present section.  By \thref{THM:MAIN} and the
  characterization in Proposition~\ref{PROP:CHARlq}, there exist
  compact sets $E$ of arbitrarily large measure satisfying item
  (c) of Proposition~\ref{PROP:CHARlq} for all $0<q<2$. Suppose
  $\{Q_j\}$ are analytic polynomials such that
  \[
    \sum_n |\widehat Q_j(n)-\widehat f(n)|^p\to0,
    \qquad
    \sup_E |Q_j|\to0,
  \]
  with $f$ belonging to the Hardy space $H^2$. By Khrushchev's theorem
  (Theorem~\ref{THM:KHRUSHMAIN}), there exists a nontrivial
  $g\in L^\infty(E)$ whose Fourier coefficients have unilateral
  polynomial decay.  For each $k\ge0$,
  \[
    0
    =
    \lim_j \int_{\T} Q_j(\zeta)\zeta^k\overline{g(\zeta)}\,dm(\zeta)
    =
    \sum_{n\ge0}\widehat f(n+k)\overline{\widehat g(n)} = \int_{\T} f(\zeta)\zeta^k\overline{g(\zeta)}\,dm(\zeta).
  \]
  By the F. and M.~Riesz theorem, $f\overline g$ belongs to the
  Hardy space $H^1$. Since $g$ is nontrivial and supported on
  $E$, this forces $f$ to vanish a.e.\ on the support of $g$,
  hence $f\equiv0$.
\end{proof}

\subsection{The main building blocks}
\label{sec:mainblocks}

Here we shall construct the main building-block, which will be the basis of our construction. 
For any integer $l>1$, consider the family of functions $\phi_l$ on $\T$ defined by
\[
\phi_l (\zeta) := \frac{l}{2\pi}\mathbbm{1}_{I(l)} \ast \frac{l}{2\pi}\mathbbm{1}_{I(l)} \ast \ldots \ast \frac{l}{2\pi}\mathbbm{1}_{I(l)}(\zeta), \qquad \zeta \in \T,
\]
where $\mathbbm{1}_{I(l)}$ denotes the indicator function of the arc $I(l)$ centered at $\zeta=1$ of length $2\pi /l$. Note that $\phi_l:\T \to [0,\infty)$ are $C^{l-2}$-smooth and have the following Fourier-decay properties:
\[
|\widehat{\phi}_l (n)| = \abs{\frac{\sin(\pi n/l)}{\pi n/l}}^l \leq
\min \biggl( 1, \biggl( \frac{l}{\pi} \biggr)^l \abs{n}^{-l} \biggr), \qquad n\neq 0
\]
and $\widehat{\phi}_l (0)=1$. Fix $0<\delta<1$, and consider the localized function
\[
  \phi_{\delta,l}(\zeta) := 
  \left\{
  \begin{array}{ll}
    (2\pi \delta)^{-1} \phi_l(\zeta^{1/\delta}) & \text{if } \abs{\zeta-1}\leq 2\pi \delta, \\
    0 & \text{if } \abs{\zeta-1}> 2\pi \delta,
  \end{array}
  \right.
\]
which is supported in an arc of length $\delta$. Given an integer $N\geq 1$, we define 
\begin{equation}\label{EQ:phi}
\phi_{N,\delta,l}(\zeta) := \phi_{\delta,k}(\zeta^N), \qquad \zeta \in \T,
\end{equation}
whose properties we shall crucially make use of, in our construction:
    \begin{enumerate}
    \item[(a)] $\widehat{\phi}_{N,\delta,l}(0)=1$,
    \item[(b)] $\widehat{\phi}_{N,\delta,l}(n)=0$ if $N \nmid n$ and 
    \[
      |\widehat{\phi}_{N,\delta,l}(Nn)| = |\widehat{\phi}_{\delta,l} (n)|
      \leq \min \biggl( 1, \biggl(\frac{l}{\pi} \biggr)^l
      \delta^{-l} \abs{n}^{-l} \biggr), \qquad n\neq 0,
    \]
    \item[(c)] $\phi_{N,\delta,l}$ is supported in $N$ arcs, each of length $2\pi \delta/N$.
\end{enumerate}
Let $U_{N,\delta}$ denote the interior of the support of $\phi_{N,\delta}$, which is the union of $N$ open arcs of length $\delta/N$. Therefore, their length does not depend on $l$. We shall construct the desired compact subset $E$ as 
\begin{equation}\label{EQ:EDEF}
E := \bigcap_{j} \T \setminus U_{N_j,\delta_j}
\end{equation}
for a suitable choices of parameters $\{N_j\}_j$, $\{\delta_j\}_j$, and also $\{l_j\}_j$ for the associated functions $\phi_{N_j, \delta_j, l_j}$, to be determined in the process.

\subsection{Sets of uniqueness for $\ell^q $ with entropic control}
\label{sec:withentropycontrol}

This section is devoted to the proof of \thref{THM:MAIN}, which
is carried out in several steps. First, we extract uniform
$\ell^q $-mass from suitable frequency blocks
associated with the construction. Next, we shall separate these
blocks in a careful arithmetic way, which will allow us to obtain
a good estimate of the Beurling--Carleson entropy of the
resulting set. Finally, a real-variable selection argument allows
us to balance entropy against $\ell^q$-divergence.

\subsubsection{Proof of \thref{THM:MAIN}}

We construct a compact set $E$ of finite Beurling--Carleson
entropy which supports no nontrivial measure with Fourier
coefficients in $\ell^q$ for any $0<q<2$. We shall
retain the notation of \eqref{EQ:EDEF}, recalling that $E$ is
determined by parameters $\{\delta_j\}_j$, $\{N_j\}_j$, and we
write
\[
  \phi_j := \phi_{N_j,\delta_j,l_j}.
\]
where $\{l_j\}_j$ is an additional parameter, at our disposal.

\paragraph{Step 1. Extracting the main frequency masses}

Note that to ensure $m(E)\ge 1-\delta$, it suffices to impose the
simple condition
\[
\sum_j \delta_j \le \delta.
\]

Fix a non-trivial $\mu\in M(E)$, and note that upon standard
modulation and normalization, we may assume that
\[
\widehat\mu(0)=\mu(E)=1.
\]

Since $\phi_j = 0$ on $\supp\mu$, for any $j\geq1$, we have
according to the properties (a)--(b) that
\[
0
=
\int_{\T}\phi_j\,d\mu
=
1+\sum_{n\neq0}
\widehat{\phi_j}(-n)\widehat\mu(N_j n).
\]
Note that
\begin{align*}
  \sum_{0<|n|\le l_j/\delta_j}
  |\widehat\mu(N_j n)|
  &\ge
    \Big|
    \sum_{0<|n|\le l_j/\delta_j}
    \widehat{\phi_j}(-n)\widehat\mu(N_j n)
    \Big| \\
  &\ge
    1-
    \sum_{|n|>l_j/\delta_j}
    |\widehat{\phi_j}(n)|\,|\widehat\mu(N_j n)| \geq 1-
    \norm{\mu} \sum_{|n|> l_j/\delta_j}
    \abs{\widehat{\phi}_j(n)},
\end{align*}
where $\norm{\mu}$ denotes the total variation norm on $\T$ of
the measure $\mu$. Using the decay of $\phi_j$ in $(c)$, we get
\begin{equation}\label{EQ:blocktail}
  \sum_{|n|>l_j/\delta_j}
  |\widehat{\phi_j}(n)|
  \leq \biggl(\frac{l_j}{\pi}\biggr)^{l_j}  \delta_j^{-l_j}
  \sum_{|n| > l_j/\delta_j} \abs{n}^{-l_j} \leq 
  C\,\pi^{-l_j}\delta_j^{-1}.
\end{equation}
We now choose $l_j \asymp \log \frac{1}{\delta_j}$ in order to
ensure that $\pi^{-l_j}\delta_j^{-1}\to0$. This implies that
\begin{equation}\label{EQ:blockmass}
  \sum_{0<|n|\le \frac{1}{\delta_j} \log \frac{1}{\delta_j}}
  |\widehat\mu(N_j n)|
  \gtrsim 1, \qquad j=1,2,3,\ldots
\end{equation}
With this estimate at hand, we may apply Hölder’s inequality with
$1<q<2$, which yields
\begin{equation}\label{EQ:blocklq}
  \sum_{0<|n|\le \frac{1}{\delta_j} \log \frac{1}{\delta_j}}
  |\widehat\mu(N_j n)|^q
  \gtrsim \biggl(
  \frac{\delta_j}{\log \frac{1}{\delta_j}}
  \biggr)^{q-1}, \qquad j=1,2,3,\ldots
\end{equation}

\paragraph{Step 2. Disjoint frequency blocks with size control}

Now, we set
\[
M_j=\left\lceil \frac{1}{\delta_j}
  \log\frac{1}{\delta_j}\right\rceil, \qquad j=1,2,3,\ldots
\]
and introduce the frequency blocks of integers
\[
\Lambda_j := \left\{ N_j n : 0<|n|\le M_j \right\}, 
\qquad j=1,2,3,\ldots
\]
of size $|\Lambda_j|=2M_j$. If the integers $N_j$ are allowed to
grow arbitrarily, then the blocks $\Lambda_j$ can easily be made
pairwise disjoint. In view of \eqref{EQ:blocklq}, this would
yield the lower bound
\[
  \sum_{n\in \mathbb{Z}} |\widehat\mu(n)|^q
  \ge
  \sum_j\sum_{n\in\Lambda_j}
  |\widehat\mu(n)|^q
  \gtrsim
  \sum_j
  \biggl(
  \frac{\delta_j}{\log\frac{1}{\delta_j}}
  \biggr)^{q-1}.
\]
However, the magnitudes of $\{N_j\}_j$ will later play a decisive
role in the estimate of the Beurling--Carleson entropy of the
underlying set $E$. We must therefore control the size of
$\{N_j\}_j$, while still maintaining the pairwise disjointness of
the blocks $\{\Lambda_j\}_j$. The following combinatorial lemma
allows us to achieve non-overlapping frequency blocks
$\Lambda_j$, without pushing the magnitude of $N_j$, much beyond
the order of $\abs{\Lambda_j}$.

\begin{lemma}
  \thlabel{LEM:blocksep}
  There exist integers $\{N_j\}_j$ such that 
  \[
    \Lambda_j\cap\Lambda_k=\varnothing
    \quad (j\neq k),
    \qquad
    N_j \lesssim |\Lambda_j|\sum_{k=1}^j |\Lambda_k|.
  \]
\end{lemma}

\begin{proof}
  Write $|\Lambda_j|=2M_j$ and construct $(N_j)$ inductively.
  Assume $N_1<\ldots<N_k$ have been chosen so that
  $\Lambda_1,\ldots,\Lambda_k$ are pairwise disjoint. Define the
  sets
  \[
    S_k:=\bigcup_{j=1}^k \Lambda_j, \qquad \Lambda_{k+1}(N):=
    \{nN: 0< |n| \leq M_{k+1}, \ n \in \mathbb{Z} \},
  \]
  where $N\geq 1$ is an integer. Note that
  \begin{equation}\label{eq:disj}
    \Lambda_{k+1}(N)\cap S_k \neq \varnothing .
  \end{equation}
  if and only if there exist an integer $m\in S_k$ and
  $1\le |n| \le M_{k+1}$ such that
  \[
    nN=m.
  \]

  There are $|\Lambda_{k+1}| = 2 M_{k+1}$ possible choices of
  $n$. For each such $n$, there are at most
  $|S_k| = \sum_{j=1}^k |\Lambda_j|$ values of $N$ for which
  $nN = m \in S_k$. Hence, in total there are not more than
  \[
    |\Lambda_{k+1}| |S_k| =|\Lambda_{k+1}|\sum_{j=1}^k|\Lambda_j|
  \]
  different values of $N$ for which
  $\Lambda_{k+1}(N)\cap S_k \neq \varnothing$.
  By the pigeon-hole principle, there must be a positive integer
  \[
    1\leq N \leq 1+|\Lambda_{k+1}|\sum_{j=1}^k|\Lambda_j|
  \]
  which fails the aforementioned property. Choosing $N_{k+1}$ to
  be the smallest such integer yields \eqref{eq:disj} and
  therefore preserves pairwise disjointness. Furthermore, we also
  have by construction that
  \[
    N_{k+1}\le 1+|\Lambda_{k+1}| \sum_{j=1}^k|\Lambda_j| \asymp
    |\Lambda_{k+1}|\sum_{j=1}^k|\Lambda_j|,
  \]
  which gives the required upper bound. The claim follows by
  induction.
\end{proof}
We remark that the integers $N_k$ may actually be chosen somewhat smaller. Indeed, let $N_k$ be the $k$-th prime exceeding $M_{k-1}$. Since $\{M_k\}_k$ are non-decreasing, this ensures that the prime numbers $\{N_k\}_k$ are distinct. Now let $j<k$ and note that $\Lambda_j \cap \Lambda_k \neq \varnothing$ if and only if
\[
nN_k=mN_j,
\qquad
1\le |n|\le M_k,\quad 1\le |m|\le M_j.
\]
But then $N_k\mid mN_j$, and since $N_k\neq N_j$ are prime, it follows that $N_k\mid m$, which is impossible due to
\[
|m|\le M_j\le M_{k-1}<N_k.
\]
Therefore $\Lambda_j\cap\Lambda_k=\varnothing$ whenever $j\neq k$, and by the prime number theorem, this choice satisfies
\[
N_k \lesssim (M_{k-1}+k)\log(M_{k-1}+k).
\]
However, we shall not use this refinement.

\paragraph{Step 3. Entropy control}

According to \thref{LEM:blocksep}, we can choose
\[
N_j \asymp M_j \sum_{k=1}^{j-1} M_k,
\]
hence upon invoking Lemma~\ref{LEM:BC}, we retrieve the entropic estimate
\[
  \mathcal E(E) \lesssim \sum_j \delta_j\log\frac{N_j}{\delta_j}
  \asymp \sum_j \delta_j\log \biggl( \sum_{k=1}^j
  \frac{1}{\delta_k} \log \frac{1}{\delta_k} \biggr).
\]
Therefore, in order to ensure finite Beurling--Carleson entropy
and divergence of the $\ell^q$-norm of
$\{\widehat{\mu}(n)\}_n$, simultaneously, it suffices to
construct decreasing positive numbers $\{\delta_j\}$ such that
\begin{equation}\label{EQ:balance}
  \sum_j
  \delta_j\log \biggl( \sum_{k=1}^j \frac{1}{\delta_k} \log
  \frac{1}{\delta_k} \biggr)
  <\infty,
  \qquad
  \sum_j
  \biggl(
  \frac{\delta_j}{\log\frac{1}{\delta_j}}
  \biggr)^{q-1}
  =\infty \quad \text{for every }1<q<2,
\end{equation}
For instance, this is achieved by the explicit choice
\[
  \delta_j = \frac{c}{j (\log j)^a}, \qquad j=1,2,3,\dots
\]
with $a > 2$, and $c>0$ small enough so that
$\sum_j \delta_j \leq \delta$. The proof is now complete.

\section{Sets of non-uniqueness in $\ell^r$}

\subsection{Constructing non-uniqueness sets in $\ell^r$}

In this subsection, we modify the building blocks introduced
earlier in order to construct sets of the form \eqref{EQ:EDEF}
which fail to be uniqueness sets for $\ell^r$ when
$1<r<2$. Our principal result here reads as follows.

\begin{prop}\thlabel{PROP:NUNIQAr}
Fix $1<r<2$ and let $0<\delta<1$. Suppose $\{\delta_j\}_{j}$ are positive numbers satisfying
\[
\sum_j \delta_j \le \delta,
\qquad 
\sum_j \delta_j^{\,r-1} < \infty.
\]
Then there exists integers $\{N_j\}_j$ such that the
corresponding set $E$ in \eqref{EQ:EDEF} satisfies
$m(E)\ge 1-\delta$ and supports a non-negative and non-trivial
function
\[
f \in L^2(E)\cap A_r(\T).
\]
\end{prop}

As will be clear from the construction, the conclusion holds whenever the sequence $\{N_j\}$ grows sufficiently rapidly. Note also that the compact sets have empty interior. 

\begin{proof}
We retain the notation of the previous subsection. Consider the functions
\[
\psi_j(\zeta)
:= \mathbbm{1}_{I(\delta_j)} * \phi_{\varepsilon_0\delta_j,l}(\zeta),
\qquad \zeta\in\T,
\]
where $\varepsilon_0>0$ chosen so that
\[
\sum_{j=1}^\infty (1+\varepsilon_0)\delta_j < 1,
\]
ensuring that $m(E)>0$. We now list some properties of $\psi_j$:

\begin{enumerate}
\item[(a)] $\psi_j\ge0$ and $\psi_j\in C^{l-1}(\T)$;
\item[(b)] $\supp{ \psi_j} \subseteq I((1+\varepsilon_0)\delta_j)$ and $\psi_j\equiv1$ on $I(\delta_j)$;
\item[(c)] $\widehat{\psi_j}(0)=\delta_j$, and for $n\ne0$,
\[
|\widehat{\psi_j}(n)|
= \frac{|\sin(\pi\delta_j n)|}{\pi|n|}
\left|\frac{\sin(\pi\varepsilon_0\delta_j n/l)}{\pi\varepsilon_0\delta_j n/l}\right|^{l}
\le 
\delta_j \min \biggl(1,\, C(l,\delta)\delta_j^{-l}|n|^{-l}\biggr).
\]
\end{enumerate}

In particular,
\begin{equation}\label{EQ:Arest}
\|\psi_j\|_{A_r}
\le C(l,\delta)\,\delta_j^{\,1-1/r},
\qquad j=1,2,3,\ldots.
\end{equation}

Define
\begin{equation}\label{EQ:hn}
h_n(\zeta)
:= \prod_{j=1}^n \bigl(1-\psi_j(\zeta^{N_j})\bigr),
\qquad \zeta\in\T.
\end{equation}

The following lemma isolates the essential mechanism.

\begin{lemma}\thlabel{LEM:AQ}
If 
\[
\sup_{n\ge1}\|h_n\|_{A_r}<\infty,
\]
then there exists a non-negative function $h\in L^2(E)$ such that $h\in A_r(\T)$.
\end{lemma}

\begin{proof}
Since $1\le r\le2$, bounded subsets of $A_r(\T)$ are weakly compact. Hence we may extract a weakly convergent subsequence $h_{n_k}\to h$ in $A_r(\T)$. Note that $h$ is also non-negative, since each $h_n$ are. Furthermore, it follows by (b) that
\[
\supp{1-\psi_j(\zeta^{N_j})}
\subseteq \T\setminus U_{N_j,\delta_j},
\]
so $\supp{h_n}\subseteq E$ for every $n$, and therefore $\supp{h}\subseteq E$. 
\end{proof}

To ensure uniform boundedness of $\|h_n\|_{A_r}$ we repeatedly use the following almost-orthogonality estimate.

\begin{lemma}[Hirschman--Katznelson, Lemma~2.a]\thlabel{LEM:AO}
Let $\psi,\phi\in A_1(\T)$ be real-valued and $r\ge1$. If $N>0$ and $\gamma>0$ satisfy
\[
\sum_{|n|\ge N} |\widehat{\phi}(n)|
\le \gamma\,\|\phi\|_{A_r},
\]
then
\[
\|\psi_N\cdot\phi\|_{A_r}
\le e^{\gamma}\,\|\psi\|_{A_r}\,\|\phi\|_{A_r},
\]
where $\psi_N(\zeta):=\psi(\zeta^N)$.
\end{lemma}

Assume
\begin{equation}\label{EQ:rsum}
\sum_j \delta_j \le \delta,
\qquad 
\sum_j \delta_j^{\,r-1}<\infty.
\end{equation}

It follows from \eqref{EQ:Arest} and the fact that $\widehat{1-\psi_j}(0)=1-\delta_j$ that there exists $c=c(l,\delta)>0$ such that
\[
\|1-\psi_j\|_{A_r}
\le \exp(c\,\delta_j^{\,r-1}).
\]

We shall construct $\{N_j\}$ inductively. Choose $N_1\ge1$ arbitrarily and set
\[
h_1 = 1-\psi_1(\zeta^{N_1}).
\]
Having chosen $N_1<\ldots<N_k$, select $N_{k+1}>N_k$ sufficiently large so that
\[
\sum_{|n|\ge N_{k+1}} |\widehat{h}_k(n)|
\le \delta_k\,\|h_k\|_{A_r}.
\]
Applying \thref{LEM:AO} with $\phi=h_k$ and $\psi=1-\psi_{k+1}$ yields
\[
\|h_{k+1}\|_{A_r}
\le e^{\delta_k}
\|1-\psi_{k+1}\|_{A_r}
\|h_k\|_{A_r}.
\]
Iterating this estimate gives
\[
\|h_{k+1}\|_{A_r}
\le 
\exp \biggl(\sum_{j=1}^k \delta_j\biggr)
\exp \biggl(c\sum_{j=1}^{k+1}\delta_j^{\,r-1}\biggr).
\]
In view of \eqref{EQ:rsum}, the right-hand side remains uniformly bounded in $k$, hence
\[
\sup_{n\ge1}\|h_n\|_{A_r}<\infty,
\]
and the result follows from \thref{LEM:AQ}.
\end{proof}

We complete this subsection by the following simple remark. Let
$g(t)$ be continuous increasing function on $[0,1]$ with the
property that
\[
\frac{g (t)}{t^{r-1}} \downarrow 0, \qquad t\downarrow 0, 
\]
for all $1<r<2$. For instance, one can take
$g(t) = \exp(-\log^2(t))$. Now if we were to choose the
parameters $\{\delta_j\}_j$ to satisfy
\[
  \sum_j \delta_j \leq \delta, \qquad \sum_j g(\delta_j) <
  \infty,
\]
then the proof actually gives a non-negative function $f\in L^2(E)$ with 
\[
\{\widehat{f}(n)\}_{n\in \mathbb{Z}} \in \bigcap_{r>1} \ell^r.
\]

\subsection{Uniqueness sets vs non-uniqueness}

We now illustrate the strength of the preceding constructions by recovering, in a streamlined manner, the following sharp main result of I.~Hirschman and Y.~Katznelson in \cite{hirschman1965sets}.

\begin{cor}[Hirschman--Katznelson, {\cite{hirschman1965sets}}] \thlabel{COR:HIRSKATZ}
Let $1<q<r<2$. Then there exist compact sets $E\subset\T$ with $m(E)$ arbitrarily close to $1$ such that:
\begin{enumerate}
\item[(a)] $E$ supports a positive measure $\mu\in A_r(\T)$.
\item[(b)] $E$ supports no non-trivial measure $\mu\in A_q(\T)$.
\end{enumerate}
\end{cor}

\begin{proof}[Proof of \thref{COR:HIRSKATZ}]
  Fix $1<q<r<2$. For any fixed $0<\delta<1$, it is a simple
  exercise in calculus to select positive numbers
  $\{\delta_j\}_j$ with the following properties:
  \begin{equation}\label{EQ:deltajselect}
    \sum_j \delta_j \leq \delta, \qquad \sum_j \delta^{r-1}_j <
    \infty, \qquad \sum_j \biggl( \frac{\delta_j}{\log
      \frac{1}{\delta_j}}\biggr)^{q-1}= \infty.
  \end{equation}
  For instance, we may choose
  $\delta_j = c j^{-\frac{1}{q-1}} (\log (1+j))^{\frac{q}{q-1} -
      \varepsilon}$, for $c,\varepsilon > 0$. When $c$ is
    sufficiently small, we have $\sum_j \delta_j \leq \delta$.

We now choose integers $\{N_j\}_j$ sufficiently large so that \thref{PROP:NUNIQAr} applies and, moreover, so that the frequency blocks
\[
\Lambda_j
:=
\left\{ nN_j : 0<|n|\le \left\lceil \frac{1}{\delta_j}\log\frac{1}{\delta_j} \right\rceil \right\},
\qquad j\ge1,
\]
are pairwise disjoint. By \thref{PROP:NUNIQAr}, the associated set $E$ supports a positive function 
$f\in L^2(E)\cap A_r(\T)$, hence a positive measure in $A_r(\T)$. On the other hand, by Step~2 of \thref{THM:MAIN}, for any $\mu\in M(E)$ (without loss of generality normalized by $\mu(E)=1$), satisfies
\[
\sum_{n\in \mathbb{Z}} |\widehat{\mu}(n)|^q
\ge
\sum_j \sum_{n\in\Lambda_j} |\widehat{\mu}(n)|^q
\gtrsim
\sum_j 
\biggl( \frac{\delta_j}{\log \frac{1}{\delta_j}}\biggr)^{q-1}
=
\infty.
\]
Thus $E$ supports no non-trivial measure in $A_q(\T)$.
\end{proof}

In fact, the construction can easily be modified to yield slightly stronger statements. 

For fixed $r\in(1,2)$ one may arrange that
\[
(a') \qquad E \text{ supports a positive measure } 
\mu \in \bigcap_{s>r} A_s(\T),
\]
while for fixed $q\in(1,2)$ one may ensure
\[
(b') \qquad E \text{ supports no non-trivial measure } 
\mu \in \bigcup_{s<q} A_s(\T).
\]

These are obtained by a minor refinements in the choice of $\{\delta_j\}_j$ in \eqref{EQ:deltajselect}, and we omit the details. However, the strengthened forms where previously obtained in the work of J. Rosenblatt and K. Shuman \cite{rosenblatt2003cyclic}, with different proofs from ours.

\subsection{Uniqueness sets via an approximation scheme}

Our proof of \thref{THM:MAIN} produces uniqueness sets for
$\ell^q $ without invoking an explicit approximation
scheme. However, by \thref{THM:CHARlq}, such a scheme must exist
whenever $E$ is a uniqueness set for $\ell^q $, namely
one ought to find functions satisfying
\[
\phi_j \in C(\T\setminus E),
\qquad
\|\phi_j - 1\|_{A_p} \to 0.
\]
At the
$N$-th stage of D.~Newman’s construction
\cite{newman1964closure}, the uniqueness sets appear as
complements of $k \gg \sqrt{N}$ intervals of length $1/N$,
forming the bases of triangular functions of height $N/k$, and
achieving the appropriate Fourier decay hinges on a delicate
result in additive number theory. Y.~Katznelson
\cite{katznelson1964sets}, introduced instead an averaging
procedure based on almost orthogonality, later refined together
with I.~Hirschman \cite{hirschman1965sets}. We sketch a variant of this approach. It yields uniqueness sets of the form
\eqref{EQ:EDEF}, though without optimal control of the parameters
$\{N_j\}_j$, which, as observed above, govern the entropic
behaviour of our sets.

For $\delta>0$ and $N\ge1$, define
\[
\psi_{\delta,N}(\zeta)
=
\delta^{-1}\mathbbm{1}_{I(\delta)}(\zeta^N),
\qquad \zeta\in\T.
\]
Then $\psi_{\delta,N}=\mathbbm{1}_{U_{\delta,N}}$ and
\[
\widehat{\psi_{\delta,N}}(0)=1,
\qquad
\widehat{\psi_{\delta,N}}(n)=0 \ \text{if } N\nmid n,
\]
while for $n\neq0$,
\[
|\widehat{\psi_{\delta,N}}(Nn)|
\lesssim
\min \bigl(1,(\delta|n|)^{-1}\bigr).
\]

Fix $0<\varepsilon<1$ and let $M>1$ be an integer to be determined later. Set $\delta=\varepsilon/M$, and for integers $N_1<\ldots<N_M$, define
\[
E(\delta,M)
:=
\bigcap_{j=1}^M (\T\setminus U_{\delta,N_j}),
\qquad
f_M
:=
\frac1M \sum_{j=1}^M \psi_{\delta,N_j}.
\]
Then we have
\[
m(E(\delta,M)) \ge 1-M\delta = 1-\varepsilon,
\qquad
f_M=0 \text{ on } E(\delta,M),
\qquad
\widehat{f_M}(0)=1.
\]
The following lemma summarizes the essence of this construction.

\begin{lemma}\thlabel{LEM:finiteapprox}
Let $p>2$ and $0<\varepsilon<1$. Then there exist integers
$M\ge1$ and $N_1<\ldots<N_M$ such that
\[
m(E(\varepsilon/M,M)) \ge 1-\varepsilon
\]
and
\[
\|f_M - 1\|_{A_p} \le \varepsilon.
\]
\end{lemma}

\begin{proof}
Since $\widehat{f_M}(0)=1$, it suffices to control the nonzero
Fourier coefficients. To this end, we may select $N_1<\ldots<N_M$ inductively so that the frequency supports
of the functions $\psi_{\delta,N_j}$ are almost disjoint
(cf.~\cite{katznelson1964sets}). As Katznelson
\cite{katznelson1964sets} we also get
\[
\|f_M-1\|_{A_p}
\le
2 M^{1/p-1} \delta^{-1}
\|\mathbbm{1}_{I(\delta)}\|_{A_p}.
\]
Using the norm-estimate $\|\mathbbm{1}_{I(\delta)}\|_{A_p}
\lesssim \delta^{1-1/p}$, we obtain
\[
\|f_M-1\|_{A_p}
\le
C M^{1/p-1/q} \varepsilon^{1/q-1},
\]
where $1/p+1/q=1$. Choosing $M=M(\varepsilon,p)$ sufficiently large yields the claim.
\end{proof}

\medskip

We now pass to the construction of a uniqueness set for
$\bigcap_{q>1} \ell^q $. Let $\{\varepsilon_j\}_j$ be
positive numbers with $\sum_j \varepsilon_j \le \varepsilon_0$.
Applying \thref{LEM:finiteapprox} with
$\varepsilon=\varepsilon_j$ produces sets $E(\delta_j,M_j)$ such
that
\[
m(E(\delta_j,M_j)) \ge 1-\varepsilon_j,
\qquad
\|f_{M_j}-1\|_{A_{p_j}} \le \varepsilon_j.
\]
Then
\[
E
=
\bigcap_{j=1}^\infty E(\delta_j,M_j)
\]
satisfies $m(E)\ge 1-\varepsilon_0$, and by
Proposition~\ref{PROP:CHARlq} it is a uniqueness set for
$\ell^q $. If one wishes to construct a uniqueness set
for $\bigcap_{q>1} \ell^q $, then one only has to
choose $p_j \downarrow 2$ and adjust the parameters $M_j$
accordingly. The details are omitted.

\subsection{One-sided uniqueness sets of uniform type}

This section is principally devoted to the proof of \thref{THM:MAIN2}. The second part of the statement in $(ii)$ requires us to establish the existence of sets uniqueness property measures with one-sided uniform bound on Fourier coefficients. The following result will be our foundation for its proof.

\begin{prop}\thlabel{PROP:SAlinfty} Let $\{\Omega(n)\}_n$ be positive real numbers with $\Omega(n) \downarrow 0$. Then for any $\{\delta_j\}_j$ positive real numbers with 
\[
\sum_j \delta_j <1,
\]
there exists positive integers $\{N_j\}_j$, such that the corresponding set in \eqref{EQ:EDEF} satisfies the following unilateral uniqueness property: whenever $\mu \in M(E)$ with 
\[
\abs{\widehat{\mu}(n)} \leq \Omega(n), \qquad n=1,2,3,\ldots,
\]
then $\mu \equiv 0$.
\end{prop}

In the proof, we shall see that the conclusion of \thref{PROP:SAlinfty} actually holds whenever $\{N_j\}_j$ grow sufficiently rapidly. Before turning to its proof, we shall derive our second main result as a consequence.

\begin{proof}[Proof of \thref{THM:MAIN2}]
Fix a number $0<\delta<1$ and let $\{\delta_j\}_j$ be positive numbers with 
\[
\sum_j \delta_j \leq \delta, \qquad \sum_j \delta^{r-1}< \infty, \qquad \text{for all} \, \, 1<r<2.
\]
Now since both the conclusions of \thref{PROP:NUNIQAr} and
\thref{PROP:SAlinfty} remain true when we choose $\{N_j\}_j$ to
grow rapidly, we can find positive integers $\{N_j\}_j$ such that
both propositions hold simultaneously. Then the associated
set $E:= \cap_j \T \setminus U_{N_j, \delta_j}$ is of Lebesgue
measure $m(E)\geq 1-\delta$, supports a non-negative function
$f\in L^2(E)$ with
\[
\{\widehat{f}(n)\}_{n\in \mathbb{Z}} \in \bigcap_{r>1} \ell^r,
\]
but there exists no non-trivial $\mu\in M(E)$ with uniform unilateral control 
\[
\abs{\widehat \mu (n)} \leq \Omega(n), \qquad n=1,2,3,\ldots \qedhere
\]
\end{proof}

We are now ready to prove our principal construction of
uniqueness sets of measures with one-sided uniform Fourier
bound. A similar construction appeared in
\cite{limani2025fourier}, but ultimately draws a great deal of
inspiration from earlier works in \cite{khrushchev1978problem}
and \cite{katznelson1964sets}.

\begin{proof}[Proof of \thref{PROP:SAlinfty}] 
  \proofpart{1}{Simultaneous approximation:}
    
  Fix positive real numbers $\{\delta_j\}_j$ with
  $\delta:=\sum_j \delta_j <1$. Let $\{\varepsilon_j\}_j$ be
  positive reals with $\varepsilon_j \to 0$, and consider smooth
  real-valued functions $\chi_{j}$ on $\T$ with the properties
  \[
    \int_{\T} \chi_j dm =0, \qquad \chi_j \equiv \log
    \varepsilon_j \, \, \, \text{off the set} \, \, \,
    I(\delta_j).
  \]
  Consider the outer function
  \[
    F_j(z) := 1- \exp\biggl( \int_{\T} \frac{\zeta+z}{\zeta-z}
    \psi_j(\zeta) dm(\zeta) \biggr), \qquad |z|<1.
  \]
  It follows that $F_j$ are analytic in the unit-disc
  $\{|z|<1\}$, and using standard properties of the Poisson kernel in
  the unit-disc, we easily verify that
  \begin{enumerate}
  \item[(a)] $F_j(0)=0$,
  \item[(b)] $\abs{F_j(\zeta) -1}\leq \varepsilon_j$ for $\zeta \notin I(\delta_j)$.
  \end{enumerate}
  Let $\{N_j\}_j$ be positive integers to be specified later, and
  consider the functions
  \[
    f_j(z) := F_j(z^{N_j}), \qquad |z|<1.
  \]
  Then it follows that 
  \[
    \abs{f_j(\zeta)-1}\leq \varepsilon_j, \qquad \zeta \in
    \{\zeta^{N_j} \notin I(\delta_j) \} = \T \setminus U_{N_j,
      \delta_j}.
  \]
  Furthermore, since $f_j(0)=0$, we get using monotonicity of $\Omega$ that 
  \[
    \sum_{n=0}^\infty | \widehat{f_j}(n) | \Omega(n) =
    \sum_{n=1}^\infty | \widehat{F_j}(n) | \Omega(N_j n) \leq
    \Omega(N_j) \norm{F_j}_{A_1}.
  \]
  Since $\Omega(n)\to 0$, and the norm $\norm{F_j}_{A_1}$ only
  depends on $\varepsilon_j, \delta_j$, we can choose the integers
  $\{N_j\}_j$, such that
  \[
    \Omega(N_j) \norm{F_j}_{A_1} \leq \varepsilon_j, \qquad j=1,2,3,\ldots
  \]
  
  \proofpart{2}{Uniqueness property:}
  Then consider the corresponding set 
  \[
    E := \bigcap_{j} \T \setminus U_{N_j, \delta_j},
  \]
  associated with the above choice of parameters
  $\{\delta_j\}_j, \{N_j\}_j$. Observe that the above
  construction yields analytic functions $\{f_j\}_j$ in
  $\{|z|<1\}$ with smooth extensions to $\T$, and which satisfy
  the simultaneous approximation phenomenon on $E$:
  \begin{equation}\label{EQ:SAl1Omega}
    \sum_{n=0}^\infty | \widehat{f_j}(n) | \Omega(n) \to 0,
    \qquad \sup_{\zeta\in E} \abs{f_j(\zeta)-1}\to 0. 
  \end{equation}

  Fix $\mu \in M(E)$ with
  $\sup_n \Omega (|n|)^{-1} \abs{\widehat{\mu}(n)}< \infty$. Using
  \eqref{EQ:SAl1Omega} and analyticity of $\{f_j\}_j$, we get for
  any integer $N\geq0$ that
  \begin{multline*}
    \abs{\widehat{\mu}(N)} = \lim_j \abs{\int_{E}
      \conj{f_j(\zeta)\zeta^N}  d\mu(\zeta)} = \lim_j
    \abs{\sum_{n=0}^\infty \conj{\widehat{f_j}(n+N) }
      \widehat{\mu}(n)} \\
    \leq \limsup_j \sum_{n=0}^\infty | \widehat{f_j}(n) |
    \Omega(n) \sup_{n\geq 0}
    \frac{\abs{\widehat{\mu}(n)}}{\Omega(n)} =0.
  \end{multline*}
  By the F. and M.~Riesz Theorem, $d\mu = h dm$ with
  $h$ belonging to the Hardy space $H^1$. Since $\supp{\mu}\subseteq E$, we see
  that $h$ vanishes on a set of positive Lebesgue measure. Since elements of $H^1$ cannot vanish on sets of positive measures, unless $h\equiv 0$, we conclude that $\mu \equiv 0$.
\end{proof}

\section{Asymmetric uniqueness phenomena for the Fourier transform}

\subsection{Fourier transform on the real line}

It turns out that our results admit direct analogues for the Fourier transform on $\R$, in the setting of $L^q(\R)$-spaces. We state only the non-periodic counterparts of
\thref{THM:MAIN} and \thref{THM:MAIN2}.

\begin{thm}\thlabel{THM:LqR}
There exist compact sets $E \subset \R$ of positive Lebesgue measure
with the following properties:
\begin{enumerate}
\item[(i)] If $\mu \in M(E)$ is nontrivial, then
\[
\int_{\R} |\widehat{\mu}(\xi)|^q\, d\xi = \infty,
\qquad 0<q<2.
\]
\item[(ii)] There exists $\nu \in M(E)$ such that for every $A>0$
there is $C(A)>0$ with
\[
|\widehat{\nu}(\xi)| \le C(A)\,\xi^{-A},
\qquad \xi>0.
\]
\end{enumerate}
\end{thm}

\begin{thm}\thlabel{THM:LqR2}
Let $\Omega:[0,\infty)\to[0,\infty)$ be non-decreasing with
$\Omega(t/2)\asymp \Omega(t)$ and $\Omega(t)\downarrow 0$.
Then there exists a compact sets $E\subset \R$ of positive Lebesgue measure such that:
\begin{enumerate}
\item[(i)] There exists a positive measure $\mu\in M(E)$ with
\[
\widehat{\mu}\in\bigcap_{r>1} L^r(\R).
\]
\item[(ii)] If $\nu\in M(E)$ satisfies
\[
|\widehat{\nu}(\xi)|\le \Omega(\xi),
\qquad \xi>0,
\]
then $\nu\equiv 0$.
\end{enumerate}
\end{thm}

Both theorems follow directly from their periodic counterparts by
identifying the sets constructed in \thref{THM:MAIN} and
\thref{THM:MAIN2} with compact subsets of $[0,1]\subset\R$ and
invoking the following simple observation of J.-P. Kahane
{\cite[Lemma~1, p.~252]{kahane1985some}}.

\begin{lemma}[J.-P. Kahane]
  Let $\Phi$ be a continuous decreasing function on $[0,\infty)$
  with $\Phi(t/2)\asymp \Phi(t)$. If $\mu$ is a compactly
  supported finite Borel measure on $\R$ such that
  \[
    |\widehat{\mu}(2\pi n)| \le \Phi(2\pi n),
    \qquad n=1,2,\ldots,
  \]
  then there exists $C(\Phi)>0$ such that
  \[
    |\widehat{\mu}(\xi)| \le C(\Phi)\Phi(\xi),
    \qquad \xi>0.
  \]
  An analogous statement holds for $\xi<0$, if we assume
  unilateral decay of negative frequencies.
\end{lemma}

Applying this lemma to the measures constructed in the periodic setting
transfers the discrete Fourier decay on $\mathbb{Z}$ to uniform decay on
$\R$, which yields \thref{THM:LqR} and \thref{THM:LqR2}.

\bibliographystyle{siam}
\bibliography{mybib}

\Addresses

\end{document}